\documentclass[12pt]{amsart}
\usepackage{amssymb,latexsym, amsmath, amscd, array, graphicx}
\numberwithin{equation}{section}
\swapnumbers
\theoremstyle{plain}
\newtheorem{thm}{Theorem}[section]
\newtheorem{theorem}[thm]{Theorem}
\newtheorem{lem}[thm]{Lemma}

\newtheorem{prop}[thm]{Proposition}
\newtheorem{cor}[thm]{Corollary}
\newtheorem{cory}[thm]{Corollary}
\newcommand\theoref{Theorem~\ref}

\newcommand\propref{Proposition~\ref}
\newcommand\corref{Corollary~\ref}

 \theoremstyle{remark}
 \theoremstyle{definition}

 \newtheorem{remarks}[thm]{Remarks}

 \def\ber{\mathfrak b}

 \DeclareMathOperator{\cat}{{\mbox{\rm cat}}}
 
 \def\cd{\protect\operatorname{cd}}

 \def\Hom{\protect\operatorname{Hom}}
 
 \def\Im{\protect\operatorname{Im}}
 \def\im{\protect\operatorname{Im}}

 \def\Ker{\protect\operatorname{Ker}}

 \def\cd{\protect\operatorname{cd}}

\newcommand \pa[2]{\frac{\partial #1}{\partial #2}}

\def\eps{\varepsilon}
\def\gf{\varphi}
\def\ga{\alpha}
\def\gb{\beta}

\def\scr{\mathcal}
\def\A{{\scr A}}

\def\Z{{\mathbb Z}}

\def\ra{{\R A}}

\def\zp{\Z[\pi]}
\def\ip{I(\pi)}

\def\pa{\partial}

\def\p{\medskip{\parindent 0pt \it Proof.\ }}
\def\wt{\widetilde}

\def\m{\medskip}
\def\nm{\noindent\medskip}
\def\ov{\overline}

\def\la{\langle}
\def\ra{\rangle}
\def\lr{\longrightarrow}
\long\def\forget#1\forgotten{} %

\date{\today}
\begin{document}
\title[Berstein-Svarc Theorem in dimension 2]
{On the Berstein-Svarc Theorem in dimension 2}

\author[A.~Dranishnikov]{Alexander N. Dranishnikov$^{1}$} %
\thanks{$^{1}$Supported by NSF, grant DMS-0604494}
\author[Yu.~Rudyak]{Yuli B. Rudyak$^{2}$}%
\thanks{$^{2}$Supported by NSF, grant DMS-0406311}
\address{Alexander N. Dranishnikov, Department of Mathematics, University
of Florida,  Little Hall, Gainesville, FL 1-10, USA}
\email{dranish@math.ufl.edu}
\address{Yuli B. Rudyak, Department of Mathematics, University
of Florida,  Little Hall, Gainesville, FL 1-10, USA}
\email{rudyak@math.ufl.edu}

\begin{abstract}
We prove that for any group $\pi$ with $\cd\pi=n$ the $n$th power of
the Berstein class is nontrivial, $\ber ^n\ne 0$.  This allows us to
prove the following Berstein-Svarc theorem for all $n$:

\noindent {\bf Theorem.} {\em For a connected complex $X$ with $\dim
X=\cat X=n$, $\ber_X^n\ne 0$ where $\ber_X$ is the Berstein class of
$X$}.

\smallskip

\nm Previously it was known for $n\ge 3$ ~\cite{Sv, Ber}.

\nm We also prove that, for every map $f: M \to N$ of degree $\pm 1$ of closed orientable manifolds, the fundamental group of $N$
is free provided that the fundamental group of $M$ is.
\end{abstract}

\maketitle

\tableofcontents

\section{Introduction}

We follow the normalization of the Lusternik-Schnirelmann category
(LS category) used in the recent monograph \cite{CLOT}, i.e. the
category of the point is equal to zero. The category of a space $X$
is denoted by $\cat X$.

\m Given a group $\pi$ we denote by $\zp$ the group ring of $\pi$ and by $\ip$
the augmentation ideal of $\pi$, i.e. the kernel of the augmentation map $\eps:
\zp \to \Z$. The cohomological dimension of $\pi$ is denoted by $\cd \pi$.

\m Consider a pointed connected $CW$-space $(X, x_0)$. Let $p: \wt  X \to X$ be
the universal covering and put $\wt X_0=p^{-1}(x_0)$. Put $\pi=\pi_1(X,x_0)$ and
let $I(\pi)$ denote the augmentation ideal of $\zp$. The exactness of the
sequence
$$
0\to H_1(\wt X, \wt X_0) \to H_0(\wt X_0) \to H_0(\wt X)
$$
yields an isomorphism $H_1(\wt X, \wt X_0)\cong I(\pi)$. So, we get
isomorphisms
$$
H^1(X,x_0;I(\pi))\cong \Hom_{\zp}(H_1(\wt X,\wt
X_0),I(\pi))\cong\Hom_{\zp}(I(\pi),I(\pi)).
$$
We denote by $\ov \ber\in H^1(X,x_0;I(\pi))$ the element that
corresponds to the identity $1_{I(\pi)}$ in the  right hand side.
Finally, we set
\begin{equation}\label{eq:bersteinclass}
\ber=\ber_X=j^*\ov \ber\in H^1(X;I(\pi))
\end{equation}
where $j:X=(X,\emptyset)\to (X,x_0)$ is the inclusion of pairs.
Consider the cup power $\ber^{n}\in H^n(X;I(\pi)\otimes\cdots
\otimes I(\pi))$ of $\ber$. The following theorem was proved by I.
Berstein~\cite[Theorem A]{Ber} and A. Svarc~\cite{Sv}, see
also~\cite[Proposition 2.1]{CLOT}

\begin{theorem}
\label{t:berstein} If $\dim X=\cat X=n\ge 3$, then $\ber^{n}\ne 0$.
{\rm (}For the case ~$n=\infty$ this means that~$\ber^{k}\ne 0$ for
all~$k$.{\rm )}
\end{theorem}

\m The proof in~\cite{Ber, Sv} is topological. It is based on the
obstruction theory and the Whitehead approach to the
Lusternik-Schnirelmann category. Thus, it cannot be extended to
$n=2$. In this paper we give an algebraic prove of the theorem for
all $n$ when $X=K(\pi,1)$. In this case the condition $\dim X=\cat
X=n$ can be replaced by a formally weaker condition $\cd\pi=n$.
Though we note that the equality $\cat K(\pi,1)=\cd\pi$ holds for
all $n$. Then using the Stallings-Swan theorem we apply our result
to extend the Berstein-Svarc theorem to the case $n=2$.

For $n=1$ the Theorem holds for trivial reasons.

\section{The Berstein class}

Here we present alternative constructions of the Berstein class and
prove its universal properties.
\m We regard $X$ as a $CW$-space with one vertex and consider the
chain complex
\[
\CD
\ldots @>>> C_1(\wt X) @>\pa_1 >> C_0(\wt X)@>>> \Z @>>> 0
\endCD
\]
of the universal covering space $\wt X$. This complex is a free
resolution of the $\zp$-module $Z$ with $C_0(X)\cong \zp$. Then $\Im
\pa_1=\ip$.

\begin{prop}\label{p:first}
The class $\ber_X$ is the cohomology class of the cocycle $f: C_1\to
\ip,\, f(c)=\pa_1(c)$.
\end{prop}

\p Consider the diagram
\[
\CD
C_2(\wt X) @>\pa_2>>C_1(\wt X) @>\pa_1 >> C_0(\wt X) @>>> 0\\
@| @| @. @.\\
C_2(\wt X, \wt X_0) @>\ov\pa_2>>C_1(\wt X; \wt X_0) @>\ov\pa_1 >> C_0(\wt X, \wt
X_0)= 0
\endCD
\]
We can compute the groups $H^*(X;\ip)$ and $H^*(X,x_0;\ip$ by applying the
functor $\Hom_{\zp}(-,\ip)$ to the first and second sequences, respectively, and
then taking the cohomology of the cochain complexes obtained. Now,
\begin{equation*}
\begin{aligned}
H_1(\wt X;\wt X_0)&=\Ker \ov\pa_1/\Im \ov \pa_2=C_1(\wt X,\wt X_0)/\Im \ov
\pa_2\\
&=C_1(\wt X)/\im \pa_2\cong \ip.
\end{aligned}
\end{equation*}
The class $\ov \ber$ is given by the quotient homomorphism
\[
C_1(\wt X, \wt X_0) \to C_1(\wt X, \wt X_0)/\im \pa_2 \to \ip.
\]

Thus, the class $\ber$ is given by the same homomorphism
$C_1(\wt X) =C_1(\wt X, \wt X_0) \to \ip$, but it is considered modulo
$$
\Im \{\delta_1: \Hom_{\zp}(C_0(\wt X),\ip) \to \Hom_{\zp}(C_1(\wt X),\ip)\}.
$$
Clearly, this homomorphism coincides with $f$.
\qed

\m The exact sequence $0\to \ip \to \zp\to \Z\to 0$ of $\zp$-modules
yields the cohomology exact sequence
\[
\CD
\cdots\, \to H^0(X;\Z) @>\delta>> H^1(X;\ip) @>>> H^1(X;\zp) \to\, \cdots
\endCD
\]

\begin{prop}\label{p:second}
We have the equality $\ber_X=\delta(1)$ where $1$ is the generator of
$H^0(X;\Z)=\Z$.
\end{prop}

\p The class $\delta(1)$ is given by any cocycle $h: C_1(\wt X) \to \ip$ such
that diagram below commutes.
\[
\CD
C_1(\wt X)@>\pa_1 >> C_0(\wt X) @= C_0(\wt X)\\
@VhVV @| @V1VV\\
\ip @>>> \zp @>>> \Z
\endCD
\]
Clearly, the map $f$ from \propref{p:first} satisfies this property.
\qed

\begin{thm}\label{t:univ} Let $X$ be a space with $\pi_1(X)=\pi$ and put
$\ber=\ber_X$. Let $A$ be a local coefficient system on $X$.

\par {\rm (i)}  For any $a\in H_1(X;A),
a\ne 0$, we have $\la a, \ber_X \ra \ne 0$ in $A\otimes_{\zp}\ip$.

\par {\rm (ii)} For every $u\in H^1(X;A)$ there exists a $\zp$-homomorphism $h:
\ip\to A$ such that $u=h_*\ber$.
\end{thm}

\p We use the interpretation of $\ber$ from
\propref{p:first}. Consider the part
\[
\CD
C_2(\wt X) @>\pa_2>> C_1(\wt X) @>\pa_1 >> C_0(X)@>>> \Z @>>> 0
\endCD
\]
of the $\zp$-resolution of $\Z$.
\par (i) Since the tensor product is right
exact, we obtain the diagram
\[
\CD
A\otimes_{\zp}C_{2}(\wt X) @>1\otimes\pa_{1}>>A\otimes_{\zp}C_1(\wt X)
@>1\otimes f>>A\otimes_{\zp}\ip @>>> 0\\
@. @. g @ VVV @.\\
@. @. A\otimes_{\zp}C_{0}(\wt X)
\endCD
\] where the row is exact.  The composition
\[ \CD A\otimes_{\zp}C_1(\wt X)@>1\otimes
f>>A\otimes_{\zp}\ip @>g>> A\otimes_{\zp}C_{0}(\wt X)
\endCD \] coincides
with~$1\otimes \pa_1$. We represent the class~$a$ by a
cycle
\[
z\in A\otimes_{\zp}C_1(\wt X).
\]
Since~$z\notin \Im(1\otimes
\pa_{2})$, we conclude that
$$
(1\otimes f)(z)\ne 0 \text{ in } A\otimes_{\zp}\ip=H_0(X;\A\otimes \ip).
$$
Thus,~$\la a,\ber \ra \ne 0$.

\par(ii) The
class $u$ is given by a homomorphism $\varphi: C_1(\wt X)\to A$ with
$\varphi \pa_2=0$. So, there exists $h: \ip=\Im
\pa_1 \to M$ with $h\pa_1=\varphi$.
\qed

\begin{cory}\label{c:product}
Suppose that there exist $u_i\in H^1(X;A_i), i=1, \ldots, n$ such
that $u_1\cup \cdots \cup u_n\ne 0$ in $H^n (X; A_1\otimes \cdots
\otimes A_n)$. Then $\ber^{n}\ne 0$.
\end{cory}

\section{The powers of the Berstein class}

Throughout this section we fix a group $\pi$ and denote
$\ber_{\pi}$ by $\ber$. We note that the
exact sequence of $\zp$-modules
\begin{equation}\label{e:basic}
\CD
0 \to I(\pi)\to\Z\pi @>\varepsilon>>\Z\to 0
\endCD
\end{equation}
stays an exact sequence of $\zp$-modules after tensoring
over $\Z$ with any $\zp$-module $M$:
$$
\CD
0\to I(\pi)\otimes M\to\Z\pi\otimes M @>\varepsilon\otimes 1>> M\to 0.
\endCD
$$
Here the $\Z$-tensor product $M\otimes N$ is taken with
the diagonal $\pi$-action.

\begin{prop}\label{p:delta} For every short
exact sequence of left $\zp$-modules $0\to N'\to N\to
N''\to 0$ and every right $\zp$-module $M$ such that the
sequence
\begin{equation}\label{e:coef}
0\to N'\otimes M\to N\otimes M\to N''\otimes M\to 0
\end{equation}
is exact, there is the formula $\delta(u\cup v)=\delta u\cup v$
where $u\in H^p(\pi; N'')$, $v\in H^q(\pi;M)$, $u\cup v\in
H^{p+q}(\pi; N'\otimes M)$ and
$$
\delta: H^{p+q}(\pi; N''\otimes M) \to H^{p+q+1}(\pi;
N'\otimes M)
$$
is the connecting homomorphism in the long exact sequence obtained
from the coefficient sequence $\eqref{e:coef}$.
\end{prop}

\p see \cite[V,(3.3)]{Bro}. \qed

\begin{cory}\label{c:delta}
The connecting homomorphism
$$
\delta:H^i(\pi;M)\to H^{i+1}(\pi;I(\pi)\otimes M)
$$
induced by the coefficient exact sequence
\begin{equation}\label{e:alpha}
\CD
0\to I(\pi)\otimes M\to \Z\pi\otimes M @>\alpha>> \Z\otimes M\to 0
\endCD
\end{equation}
has the form $\delta(u)=\ber \cup u$.
\end{cory}

\p Note that every $u\in H^1(\pi;\Z\otimes M)$ can be presented as
$1\cup u$ where $1\in H^0(\pi;\Z)$. Since by Proposition
\ref{p:second} $\delta(1)=\ber$, the result follows from
\propref{p:delta}. \qed

\m Let $I=I(\pi)$.
Consider the exact sequence
$$
\CD
0\to I^{k+1} @>{\beta_k}>> \Z\pi\otimes I^k @>\alpha_k>> I^k\to 0
\endCD
$$
obtained from the sequence \eqref{e:basic} via tensoring by $I^k$.

\begin{lem}\label{l:free} {\rm (i)} The $\zp$-module $\zp^{\otimes k}=\zp\otimes
\cdots \otimes\zp$ is free for all $k$;

\par{\rm (ii)} the $\zp$-module $\zp\otimes I^k$ is projective for all $k$.
\end{lem}

\p (i) Clearly, $\zp\otimes \cdots \otimes \zp=\Z[\pi\times \cdots \times \pi]$
is a free
$\Z[\pi\times \cdots \times \pi]$-module. Regarding $\pi$ as the diagonal
subgroup of
$\pi\times \cdots \times \pi$, we conclude that $\zp^{\otimes k}$ is a free
$\zp$-module. (If $H$ is a subgroup of $G$ then $\Z[G]$ is a free
$\Z[H]$-modulee
with the base $G/H$.)

(ii) We tensor the sequence \eqref{e:basic} by $I^{k-1}\otimes \zp$ and get the
exact sequence
\[
0\to I^k\otimes \zp\to I^{k-1}\otimes \zp\otimes \zp \to I^{k-1}\otimes \zp\to
0.
\]
Now induct on $k$, using (i) as the basis of induction.
\qed

\begin{theorem}\label{t:power}
If $\cd \pi =n$, then $\ber^{n}\ne 0$.
\end{theorem}

\p For $n=1$ this is trivial. Let $n>1$.
In view of Lemma \ref{l:free} there is a projective resolution
\begin{equation}\label{e:res}
C_n \stackrel{\partial_n}\lr C_{n-1} \lr \cdots \lr C_1 \lr C_0 \lr \Z\lr 0
\end{equation}
such that $C_k=\Z\pi\otimes I^k$ for all $k$ and with
$$
\CD
\partial_k=\beta_{k-1}\circ\alpha_k:\Z\pi\otimes I^k @>\alpha_k>> \Z\otimes I^k
@>\beta_{k-1}>> \Z\pi\otimes
I^{k-1}.
\endCD
$$

Claim: $\ber^{k}=[\alpha_k]$. We apply the induction on $k$.

For $k=1$, note that in \propref{p:first} we can construct $\ber$
using any $\zp$-resolution of $\Z$ provided $C_0=\zp$. Hence,
$\ber=[\ga_1]$. Furthermore, it is easy to see that
$\delta([\ga_k])=[\ga_{k+1}]$, and hence $\ber^{k}=[\ga_k]$ by
\corref{c:delta}

\m Now we prove that $\ber^n\ne 0$. Suppose the contrary. Then
$\alpha_n=\gf\circ\partial_n=\gf\circ\gb_{n-1}\circ\ga_n$ for some
$\gf:C_{n-1}\to I^n$.
Since $\ga_n$ is an epimorphism, we conclude that $\gf$ yields a $\zp$-splitting
of the sequence
$$
\CD
0\to I^n @>\beta_{n-1}>> \Z\pi\otimes I^{n-1} @>\alpha_{n-1}>> I^{n-1}\to 0.
\endCD
$$
This implies that $I^{n-1}$ is projective. Thus the above resolution
can be shorten by one. Therefore $\cd\pi\le n-1$. Contradiction.
\qed

The following universal property was stated as \cite[Proposition
3.4]{Sv} without proof.

\begin{cor}\label{c:svarz} Given a $\zp$-module $A$ and a class $a\in
H^k(\pi;A)$,
there exists a $\zp$-homomorphism $\mu: I^k \to A$ such that
$\mu_*(\ber)=a$.
\end{cor}

\begin{proof} Consider the resolution \eqref{e:res}. The class $a$
is given by a homomorphism $h:\zp\otimes I^k\to A$. Since $h$ is a
cocycle, $h\circ \pa_{k+1}=0$, we conclude that $h=\mu \circ\ga_k$
for some $\mu: I^k\to A$.
\end{proof}

\begin{remarks} 1. Of course, \corref{c:svarz} implies \theoref{t:power}.
However, we decided to prove \theoref{t:power} here since
Svarc~\cite{Sv} stated it without proof.

2. For $n\ne 2$ \theoref{t:power} follows from \theoref{t:berstein}
in view of the equality $\cat K(\pi,1)=n$ and existence of an
$n$-dimensional $CW$-space $K(\pi,1)$, \cite{EG}. So, the new point
is the case $n=2$.
\end{remarks}

\section{Applications to the Lusternik--Schnirelmann category}

Now we can prove the Berstein-Svarc Theorem for $n=2$.

\begin{prop}\label{p:cat2}
If $X$ is a connected $2$-dimensional $CW$-complex with $\pi_1(X)$ free,
then $\cat X \le 1$.
\end{prop}

\p Since $\pi_1(X)$ is free, we can construct a space
$$
Y=(\vee_{\ga}S^1_{\ga})\bigvee(\vee_{\gb}S^2_{\gb})
$$
and a map $f: Y \to X$ such that $f_*: \pi_1(Y)\to \pi_1(X)$ is an
isomorphism and $f_*: \pi_2Y)\to \pi_2(X)$ is an epimorphism. Since the
homotopy fiber of $f$ is simply connected, there exists a map (homotopy
section) $s: X \to Y$ such that $fs\cong 1_X$. Hence, $\cat X \le \cat
Y=1$.
\qed

\begin{thm}\label{t:main}
Given a connected $2$-dimensional $CW$-complex $X$, the equality
$\cat X=2$ holds if and only if $\ber_X^{2}\ne 0$.
\end{thm}

\p Let $\pi$ denote $\pi_1(X)$. First, note that $\cat X \le \dim
X=2$. Clearly, if $\ber_X^{2}\ne 0$ then $\cat X \ge 2$, and so
$\cat X=2$. Furthermore, if $\cat X=2$ then $\pi$ is not free by
\propref{p:cat2}. Hence, according to Stallings~\cite{Stal} and
Swan~\cite{Swan}, $\cd\pi>1$. So, by \theoref{t:power},
$\ber_{\pi}^{2}\ne 0$. Let $f: X \to K(\pi,1)$ be a map that induces
an isomorphism on the fundamental groups. Then $f^*: H^2(\pi;M)\to
H^2(X;M)$ is injective for any $\zp$-module $M$, and the theorem
follows since $f^*(\ber_{\pi})=\ber_X$. \qed

\section{Applications to maps of degree 1}

Let $X$ be a closed oriented $n$-dimensional manifold with the
fundamental class $[X]$. Then for a local coefficient system $G$ on
$X$ we have the Poincar\'e duality isomorphism $\Delta: H^i(X;G) \to
H_{n-i}(X;G)$, $\Delta(x)=x\cap [X]$, see e.g. ~\cite{Bre}. Given a
map $f: M^n \to N^n$ of two closed manifolds and a local coefficient
system $G$ on $N$, we define the homomorphism $f^!:
H^*(M;f^*(G))\to^*(N;G) $ by setting
\[
f^!(x)=\Delta_N^{-1}f_*\Delta_M(x)
\]\

For future needs we state the following known proposition.

\begin{prop}\label{p:degree} The index $[\pi_1(N): f_*(\pi_1(M))]$
does not exceed $|\deg f|$ provided that $\deg f \ne 0$, and
$f^!f^*(x)=(\deg f) x$ for any $x\in H^*(M, f^*(G))$.
\end{prop}

\begin{proof} The first claim follows since $f$ passes through the
covering map $\wt N \to N$ corresponding to the subgroup
$f_*(\pi_1(M))$ of $\pi_1(N)$. For the second claim, we have
\begin{equation*}
\begin{aligned}
f^!(f^*x)&=\Delta_N^{-1}f_*\Delta_M(f^*x)=\Delta_N^{-1}f_*((f^*x\cap
M)=\Delta_N^{-1}(x\cap f_*{M})\\
&=\deg f(\Delta_N^{-1}(x)\cap N=(\deg f) x
\end{aligned}
\end{equation*}
\end{proof}

\begin{thm} If $f:M \to N$ is a map of degree $1$ and $\pi_1(M)$ is
free then $\pi_1(N)$ is free.
\end{thm}

\begin{proof} Let $\pi=\pi_1(M), \tau=\pi_1(N)$ and assume that $\tau$ is not
free. Then $\cd \tau>1$ according to Stallings~\cite{Stal} and
Swan~\cite{Swan}. So, $\ber_{\tau}^2\ne 0$ by \theoref{t:power}.
Hence $\ber_N^2\ne 0$, and so $f^*\ber_N^2\ne 0$ by
\propref{p:degree}. So, $\ber_M^2\ne 0$ by \corref{c:product}.
Hence, $\ber_{\pi}^2\ne 0$, and so $\cd \pi>1$. Thus, $\pi$ is not
free.
\end{proof}

\begin{remarks} 1. Note that $f_*;\pi_1(M) \to \pi_1(N)$ is an
epimorphism for any map of degree 1 because of \propref{p:degree}

2. Every epimorphism $h: F_r\to F_s$ of free finitely generated
groups can be realized as $h=f_*:\pi_1(M)\to \pi_1(N)$ for a map $f:
M \to N$ of degree 1. Indeed, take $M$ and $N$ to be the connected
sums of $r$ and $s$ copies of $S^1\times S^n$, respectively, and put
$f$ to be the map that shrinks $r-s$ copies of $S^1\times S^n$.
\end{remarks}

\end{document}